\documentclass[12pt]{article}

\usepackage{geometry}
\usepackage{setspace}

\usepackage{lmodern,enumerate,rotating,anyfontsize}
\usepackage{amsmath,amsfonts,amssymb,subfigure,mathtools,blkarray,soul,etoolbox}
\usepackage{tikz-cd}
\usetikzlibrary{cd}
\usetikzlibrary{shapes,shapes.geometric,arrows,fit,calc,positioning,automata}

\usepackage{stmaryrd}
\SetSymbolFont{stmry}{bold}{U}{stmry}{m}{n}

\DeclareMathOperator{\CCa}{CC_A}
\DeclareMathOperator{\CCaLeft}{CC_A^{\leftarrow_\ep}}

\DeclareMathOperator{\spec}{spec}
\DeclareMathOperator{\UR}{UR}
\DeclareMathOperator{\Reach}{Reach}

\allowdisplaybreaks

\newcommand{\LM}{\mathcal{L}}

\newcommand{\dt}{\delta}

\newcommand{\ep}{\epsilon}

\newcommand{\Scal}{\mathcal{S}}

\newcommand{\Sig}{\Sigma}
\newcommand{\s}{\sigma}

\newcommand{\Mt}{\mathcal{M}}

\newcommand{\llb}{\llbracket}
\newcommand{\rrb}{\rrbracket}

\newtheorem{theorem}{Theorem}[section]

\newtheorem{definition}{Definition}
\newtheorem{proposition}[theorem]{Proposition}
\newtheorem{assumption}{Assumption}
\newtheorem{example}[theorem]{Example}

\newtheorem{remark}{Remark}

\usepackage{times,amssymb,amsfonts,mathtools,graphicx,tikz,algorithm,algorithmic,url,hhline,booktabs}
\usetikzlibrary{automata,arrows,positioning}

\usetikzlibrary{shadows}

\makeatletter
\let\NAT@parse\undefined
\makeatother
\usepackage[colorlinks,linkcolor=blue,anchorcolor=blue,citecolor=green,urlcolor=cyan,hyperindex]{hyperref}

\newcommand{\PSPACE}{\mathsf{PSPACE}}

\newcommand{\NPSPACE}{\mathsf{NPSPACE}}
\newcommand{\coNPSPACE}{\mathsf{coNPSPACE}}
\newcommand{\EXPTIME}{\mathsf{EXPTIME}}
\newcommand{\PTIME}{\mathsf{P}}

\newcommand{\NL}{\mathsf{NL}}
\newcommand{\coNL}{\mathsf{coNL}}

\usetikzlibrary{automata,arrows,positioning}
\usetikzlibrary{petri}

\newcounter{enumi_saved}
\newenvironment{myenumerate} {
\begin{enumerate}[(i)]\setcounter{enumi}{\value{enumi_saved}}}
{\setcounter{enumi_saved}{\value{enumi}}\end{enumerate}}

\tikzset{elliptic state/.style={draw,ellipse}}

\tikzset{
node distance=3cm, 
every state/.style={thick, fill=gray!10}, 
initial text=$ $, 
}

\definecolor{green}{rgb}{0.1,0.7,0.1}

\usepackage{tcolorbox}
\usepackage{tabularx}
\usepackage{array}
\usepackage{colortbl}
\tcbuselibrary{skins}

\title{New methods for verifying strong periodic detectability and strong periodic D-detectability of discrete-event systems\footnote{The materials in the paper were originally submitted to Journal on April 5, 2020.}}

\author{Kuize Zhang\\
{\small Control Systems Group, Technical University of Berlin, 10587 Berlin, Germany}\\
{\small kuize.zhang@campus.tu-berlin.de}
	}

\begin{document}

\date{}

\maketitle

{\bf Abstract}
	In this paper, in discrete-event systems modeled by finite-state automata (FSAs), we show new thinking 
	on the tools of detector and concurrent composition and derive two new algorithms for verifying strong periodic 
	detectability (SPD) without any assumption that run in $\NL$; we also reconsider the tool of observer and
	derive a new algorithm for verifying strong periodic D-detectability (SPDD) without any assumption that runs
	in $\PSPACE$. These results
	strengthen the $\NL$ upper bound on verifying SPD and the $\PSPACE$ upper bound on verifying SPDD
	for deadlock-free and divergence-free FSAs in the literature.


\section{Introduction}

\section{Introduction}

\subsection{Background}

\emph{Detectability} is a basic property of partially-observed dynamical systems: when it holds one
can use an observed output/label sequence produced by a system to reconstruct its states
\cite{Shu2007Detectability_DES,Shu2011GDetectabilityDES,Zhang2019KDelayStrDetDES,Zhang2020DetPNFA}.
This property plays a fundamental role in many related control problems such as observer design and controller 
synthesis. Detectability is quite related to another fundamental property diagnosability which implies
occurrences of all faulty events could be detected after sufficiently many occurrences of subsequent events
\cite{Sampath1995DiagnosabilityDES}. Recently, strong detectability and diagnosability have been unified
into one mathematical framework \cite{Zhang2021UnifyingDetDiagPred} in \emph{discrete-event systems} (DESs)
modeled by \emph{finite-state automata} (FSAs). On the other hand, detectability is
strongly related to many cyber-security properties. For example, the property of opacity, which has been
originally proposed to describe information flow security  in computer science in the early 2000s 
\cite{Mazare2004Opacity} can be seen as the absence of detectability.

For \emph{discrete-event systems} (DESs) modeled by \emph{finite-state automata} (FSAs),
the verification problems for different definitions of detectability
have been widely studied
\cite{Shu2007Detectability_DES,Shu2011GDetectabilityDES,Zhang2019KDelayStrDetDES,Zhang2017PSPACEHardnessWeakDetectabilityDES,Masopust2018ComplexityDetectabilityDES,Zhang2020DetPNFA,Balun2021ComplexitySPDdetJournal,Zhang2021UnifyingDetDiagPred},
in which several complexity lower bounds and upper bounds for these problems were obtained,
but most of the upper bounds
depend on two fundamental assumptions that a system is deadlock-free and divergence-free.
These requirements are collected in Assumption~\ref{assum1_Det_PN}: when it holds, a system will always
run and generate an infinitely long label/output sequence. The first verification algorithm for detectability 
of DESs that does not depend on Assumption~\ref{assum1_Det_PN} was given by us in \cite{Zhang2019KDelayStrDetDES}
by developing a technique called \emph{concurrent composition}, which was used to verify negation of strong
detectability. In the current paper, we further develop new methods to obtain verification algorithms
for strong periodic detectability and strong periodic D-detectability that do not depend on any assumption,
so that complexity upper bounds will be obtained for all FSAs.

We recall basic complexity results used in the paper (see \cite{Sipser2006TheoryofComputation,Immerman1988NLcoNL}).
The symbols $\NL$, $\PTIME$, $\PSPACE$,
$\NPSPACE$, and $\EXPTIME$ denote the sets of problems solvable in nondeterministic logarithmic space, 
polynomial time, polynomial space, nondeterministic polynomial space, and exponential time,
respectively. $\coNL$ and $\coNPSPACE$ denote the sets of problems whose complements
belong to $\NL$ and $\NPSPACE$, respectively. It is known that $\NL\subset\PTIME\subset\PSPACE\subset
\EXPTIME$, $\NL=\coNL$, and $\PSPACE=\NPSPACE=\coNPSPACE$. It is also known that $\NL\subsetneq\PSPACE$ and 
$\PTIME\subsetneq\EXPTIME$, but whether the rest of these containments are strict are long-standing open
questions. It is widely conjectured all the other containments are strict.
A problem $A$ is \emph{$\NL$-hard} (resp., \emph{$\PSPACE$-hard}) if every problem in $\NL$ (resp., 
$\PSPACE$) is log space (resp., polynomial time) reducible to $A$, for which A has an $\NL$ (resp., $\PSPACE$)
lower bound. A problem $A$ is \emph{$\NL$-complete}
(resp., \emph{$\PSPACE$-complete}) if $A$ belongs to $\NL$ (resp., $\PSPACE$) and is $\NL$-hard
(resp., $\PSPACE$-hard). If a problem $A$ belongs to $\NL$ (resp., $\PSPACE$), then $A$ has an $\NL$
(resp., $\PSPACE$) upper bound. In this paper, we sometimes say a detectability property has an
$\NL$ ($\PSPACE$) upper bound for short, which means that the problem of verifying the property
in FSAs belongs to $\NL$ ($\PSPACE$).

\subsection{Literature review on verification of detectability in FSAs}

\textbf{Results based on Assumption~\ref{assum1_Det_PN}}
In \cite{Shu2007Detectability_DES}, by using an \emph{observer}\footnote{i.e., the powerset construction used for 
determinizing nondeterministic finite automata with $\ep$-transitions \cite{Sipser2006TheoryofComputation}} method,
exponential-time algorithms were given to verify four notions of detectability: strong (periodic) detectability
and weak (periodic) detectability. Strong detectability means that there is a time delay $k$, for \emph{each}
infinite-length event sequence $s$ generated
by an FSA, every prefix of the label/output sequence of $s$ of length greater than $k$ allows reconstructing
the current state. Weak detectability relaxes strong detectability by replacing \emph{each} to \emph{some}.
Weak detectability is strictly weaker than strong detectability. 
Strong periodic detectability implies that at any time, after some observation time
delay no greater than a given value, the system states can be determined along 
\emph{each} infinite-length transition sequence also by observing the corresponding output sequence.
Weak periodic detectability relaxes strong periodic detectability also by changing \emph{each} to
\emph{some}. Later in \cite{Shu2011GDetectabilityDES}, by using a \emph{detector} (obtained from an observer
by splitting all its states into subsets of cardinality $2$) method, 
polynomial-time algorithms were designed for verifying strong (periodic)
detectability. The problem of verifying weak (periodic) detectability of FSAs was proven to be
PSPACE-complete \cite{Zhang2017PSPACEHardnessWeakDetectabilityDES}
and the problem of verifying strong (periodic) detectability was proven to be $\NL$-complete
\cite{Masopust2018ComplexityDetectabilityDES}.

In order to make detectability adapt to more scenarios, one can weaken detectability
to D-detectability
in the sense of not exactly determining the states but making sure that the 
states cannot contain both states of any pair of states that are previously specified 
\cite{Shu2011GDetectabilityDES}. All above notions
of detectability, including strong/weak detectability and strong/weak periodic detectability,
can be extended to their D-versions. For example, strong D-detectability can be verified in polynomial
time \cite{Shu2011GDetectabilityDES}, while verifying strong periodic D-detectability 
is $\PSPACE$-complete \cite{Balun2021ComplexitySPDdetJournal}. 

Note that all the above complexity upper bounds were obtained by the verification algorithms designed in 
\cite{Shu2007Detectability_DES,Shu2011GDetectabilityDES} based on Assumption~\ref{assum1_Det_PN}. 
For an FSA that does not satisfy Assumption~\ref{assum1_Det_PN}, the algorithms may not return a correct 
answer. In \cite[Remark 2]{Zhang2020DetPNFA}, we had given a counterexample to show that neither the observer 
method \cite{Shu2007Detectability_DES} nor the detector method \cite{Shu2011GDetectabilityDES} correctly
verifies its strong detectability. Later in Remark~\ref{rem1_spdet_FSA} and
Remark~\ref{rem2_spdet_FSA}, we will give counterexamples to show that neither of the two methods correctly
verifies their strong periodic detectability and strong periodic D-detectability.

\textbf{Results which do not depend on assumptions}
The two fundamental assumptions shown in Assumption~\ref{assum1_Det_PN} was for the first time removed 
by us in \cite{Zhang2019KDelayStrDetDES,Zhang2020DetPNFA}
by developing a \emph{concurrent-composition} method and verifying \emph{negation} of strong detectability.
In \cite{Zhang2020DetPNFA}, weak detectability was also verified without any assumption. Later in 
\cite{Zhang2021UnifyingDetDiagPred}, an $\NL$ upper bound was given for the verification problem of 
strong detectability based on the concurrent-composition method. In addition, decentralized settings
of strong detectability, diagnosability, and predictability were unified into one mathematical framework
\cite{Zhang2021UnifyingDetDiagPred}. In \cite{Zhang2020RevisitingDelayStrDetDES}, strong D-detectability
was verified in polynomial time also by the concurrent-composition method.

\subsection{Contribution of the paper}

The contributions of the paper are as follows:
\begin{enumerate}
	\item We use the detector and concurrent composition to derive two new algorithms for verifying
		strong periodic detectability of FSAs without any assumption, where both algorithms imply
		an $\NL$ upper bound for strong periodic detectability,
		which strengthens the $\NL$ upper bound given in \cite{Masopust2018ComplexityDetectabilityDES} 
		under Assumption~\ref{assum1_Det_PN}.
	\item We use the observer to derive a new algorithm for verifying strong periodic D-detectability of FSAs
		without any assumption, where
		the algorithm implies a $\PSPACE$ upper bound for strong periodic D-detectability,
		which strengthens the $\PSPACE$ upper bound given in \cite{Balun2021ComplexitySPDdetJournal}
		under Assumption~\ref{assum1_Det_PN}. See Tab.~\ref{tab1:ComplexityDetectabilityFSA} for a collection of related results.
\end{enumerate}

Differently from verifying strong periodic detectability itself in 
\cite{Shu2011GDetectabilityDES,Masopust2018ComplexityDetectabilityDES},
we verify its negation. Following such an opposite way, for an FSA,
we obtain two conditions on its observer such that at least one of them holds exactly violates its
strong periodic detectability. Thus an exponential-time algorithm for verifying strong periodic detectability
is obtained without any assumption (Theorem~\ref{thm3_spdet_FSA}). Furthermore, by developing 
a new relationship between the notions of observer and detector (Proposition~\ref{prop5_spdet_FSA}),
the exponential-time algorithm is reformulated by a detector, resulting in a polynomial-time
verification algorithm (Theorem~\ref{thm4_spdet_FSA}). Thus, an $\NL$ upper bound naturally follows
from the polynomial-time verification algorithm
(Theorem~\ref{thm5_spdet_FSA}). On the other hand, by developing more relationships between 
the notions of observer, detector, and concurrent composition (Proposition~\ref{prop8_spdet_FSA}),
we construct a variant of the concurrent composition by using which strong periodic detectability
can also be verified in $\NL$ (Theorem~\ref{thm8_spdet_FSA}). Similarly, we also obtain a polynomial-space
verification algorithm for strong periodic D-detectability by verifying its negation.

The remainder is structured as follows. In Section~\ref{sec2:pre}, basic notation and definitions in
FSAs are introduced. In Section~\ref{sec3:mainresult}, the main results are shown. Section~\ref{sec4:conc}
ends up this paper with short conclusion.

\begin{table}
	\centering
	\begin{tabular}{c|c}
		\hline\hline
		strong detectability & strong D-detectability\\
		\hline
		\begin{tabular}{c}
			$\NL$ (\cite{Zhang2021UnifyingDetDiagPred})\\
			$\NL$-complete* (\cite{Masopust2018ComplexityDetectabilityDES})
		\end{tabular}
		&
		\begin{tabular}{c}
			$\PTIME$ (\cite{Zhang2020RevisitingDelayStrDetDES})\\
			$\NL$-complete* (\cite{Balun2021ComplexitySPDdetJournal})
		\end{tabular}\\\hline\hline
		strong periodic detectability & strong periodic D-detectability\\
		\hline
		\begin{tabular}{c}
			$\NL$ (Thm.~\ref{thm5_spdet_FSA})\\
			$\NL$-complete* (\cite{Masopust2018ComplexityDetectabilityDES})
		\end{tabular}
		&
		\begin{tabular}{c}
			$\PSPACE$ (Thm.~\ref{thm7_spdet_FSA})\\
			$\PSPACE$-complete* (\cite{Balun2021ComplexitySPDdetJournal})
		\end{tabular}\\\hline\hline
		weak detectability & weak periodic detectability\\\hline
			$\PSPACE$-complete* (\cite{Zhang2017PSPACEHardnessWeakDetectabilityDES})
		&
			$\PSPACE$-complete* (\cite{Zhang2017PSPACEHardnessWeakDetectabilityDES})
		\\\hline
	\end{tabular}
	\caption{Complexity results for verifying different definitions of detectability in FSAs,
	where $^*$ means that the $\NL$ and $\PSPACE$ upper bounds only apply to FSAs satisfying 
	Assumption~\ref{assum1_Det_PN}.}
	\label{tab1:ComplexityDetectabilityFSA}
\end{table}

\section{Preliminaries}\label{sec2:pre}

We introduce necessary notion that will be used throughout this paper.
For a finite alphabet $\Sig$, $\Sig^*$ and $\Sig^{\omega}$ are used to denote the set of finite sequences
(called \emph{words}) of elements of $\Sig$ including the empty word $\epsilon$
and the set of infinite sequences (called \emph{configurations}) of elements of $\Sig$,
respectively. $\Sig^{+}:=\Sig^*\setminus\{\epsilon\}$.
For a word $s\in \Sig^*$,
$|s|$ stands for its length. 
For $s\in \Sig^+$ and natural number $k$, $s^k$ and $s^{\omega}$ denote the \emph{concatenations}
of $k$-copies and infinitely many copies of $s$, respectively.
For a word (configuration) $s\in \Sig^*(\Sig^{\omega})$, a word $s'\in \Sig^*$ is called a \emph{prefix} of $s$,
denoted as $s'\sqsubset s$,
if there exists another word (configuration) $s''\in \Sig^*(\Sig^{\omega})$ such that $s=s's''$.
For two natural numbers $i\le j$, $\llb i,j\rrb$ denotes the set of all integers no less than $i$ and no
greater than $j$;
and for a set $S$, $|S|$ its cardinality and $2^S$ its power set. As usual, a \emph{singleton}
is defined by a set of cardinality $1$. $\subset$ denotes the subset relation.

A DES modeled by an  FSA is a sextuple
\begin{equation}\label{DES_Det_NDES}
	\Scal=(X,T,X_0,\dt,\Sig,\ell),
\end{equation}
where $X$ is a finite set of \emph{states}, $T$ a finite set of \emph{events},
$X_0\subset X$ a set of \emph{initial states},
$\dt\subset X\times T\times X$ a \emph{transition relation}, $\Sig$ a finite set of \emph{outputs (labels)},
and $\ell:T\to\Sig\cup\{\epsilon\}$
a \emph{labeling function}. 
$\ell$ can be recursively extended to $\ell:T^*\cup T^{\omega}\to\Sig^*\cup\Sig^{\omega}$
as $\ell(t_1t_2\dots)=\ell(t_1)\ell(t_2)\dots$ and particularly $\ell(\ep)=\ep$.
The event set $T$ can be rewritten as disjoint union
of \emph{observable} event set $T_{o}=\{t\in T|\ell(t)\in\Sig\}$ and \emph{unobservable} event set $T_{uo}=
\{t\in T|\ell(t)=\ep\}$.
Transition relation $\dt$ is recursively extended to $\dt\subset X\times T^*\times X$ in the usual way.
We call a transition with an observable (unobservable)
event an \emph{observable (unobservable) transition}. We also denote a transition sequence $(x,s,x')\in\dt$
by $x\xrightarrow[]{s}x'$, where $x,x'\in X$, $s\in T^*$.
For $x\in X$ and $s\in T^+$, $(x,s,x)$ is called a \emph{transition 
cycle} if $(x,s,x)\in\dt$. An \emph{observable} (resp., \emph{unobservable}) transition cycle is defined 
by a transition cycle with at least one (resp., with no) observable transition.
Automaton $\Scal$ is called \emph{deterministic} if $|X_0|=1$ and 
for all $x,x',x''\in X$ and $t\in T$, $(x,t,x'),(x,t,x'')\in\dt$ imply $x'=x''$.
For deterministic $\Scal$, for all $x\in X$ and all $s\in T^*$, we also denote the unique state $x'\in X$
(if any) satisfying $x\xrightarrow[]{s}x'$ by $\dt(x,s)$.
For two states $x,x'\in X$, we say \emph{$x'$ is reachable from $x$} if there is $s\in T^+$
such that $(x,s,x')\in\dt$; we say \emph{$x'$ is reachable} if either $x'\in X_0$ or $x'$ is 
reachable from some initial state. Analogously, reachability from a set of states to a state and vice versa
could also be defined. Particularly, we call a transition cycle \emph{reachable} if it is reachable from 
some initial state.

For each $\s\in\Sig^*$, we denote by $\Mt({\Scal},\s)$ the \emph{current-state estimate},
i.e., the set of
states that the system can be in after $\s$ has been observed, i.e., 
$\Mt({\Scal},\s):=\{x\in X|(\exists x_0\in X_0)(\exists s\in T^*)[
(\ell(s)=\s)\wedge(x_0\xrightarrow[]{s}x)]\}$.
We use $L(\Scal)=\{s\in T^*|(\exists x_0\in X_0)(\exists x\in X)[x_0\xrightarrow[]
{s}x]\}$ to denote the set of finite-length event sequences generated by $\Scal$,
we also use $L^{\omega}(\Scal)=\{t_1t_2\dots$$\in T^{\omega}|(\exists x_0\in X_0)
(\exists x_1,x_2,\dots$$\in X)[x_0\xrightarrow[]{t_1}x_1\xrightarrow[]{t_2}\cdots]\}$ to 
denote the set of infinite-length event sequences generated by $\Scal$.
Analogously, we use $\LM({\Scal})$ denotes the \emph{language generated} by $\Scal$,
i.e., $\LM({\Scal}):=\{\s\in\Sig^*|\Mt({\Scal},\s)\ne\emptyset\}$,
we also use $\LM^{\omega}({\Scal})$ to denote the $\omega$-\emph{language generated} by $\Scal$,
i.e., $\LM^{\omega}(\mathcal{ S}):=\{\s\in\Sigma^{\omega}|(\exists s\in L^{\omega}(\Scal)
[\ell(s)=\s]\}$.

For a state $x\in X$, its \emph{unobservable reach} is defined by
$\UR(x):=\{x'\in X|(\exists s\in (T_{uo})^*)[(x,s,x')\in\dt]\}$.
For a subset $X'\subset X$, $\UR(X')=\bigcup_{x\in X'}\UR(x)$. Hence $\UR(X_0)=\Mt(\Scal,\ep)$.
For a state $x\in X$,
its \emph{observable reach under $\s\in \Sig$} is defined by $\Reach_\s(x):=\{
x'\in X|(\exists t\in T)[((x,t,x')\in\dt)\wedge(\s=\ell(t))]\}$. Analogously, for a subset
$X'\subset X$, $\Reach_{\s}(X')=\bigcup_{x\in X'}\Reach_\s(x)$.

The following two assumptions are commonly used in detectability studies 
(cf. \cite{Shu2007Detectability_DES,Shu2011GDetectabilityDES,Masopust2018ComplexityDetectabilityDES,Balun2021ComplexitySPDdetJournal}),
but are not needed in the current paper based on our new thinking of the tools of observer,
detector, and concurrent composition.

\begin{assumption}\label{assum1_Det_PN}
	An FSA $\Scal$ as in \eqref{DES_Det_NDES} satisfies
	\begin{enumerate}[(A)]
		\item\label{item11_Det_PN}
			$\mathcal S$ is \emph{deadlock-free}, i.e.,
			for each reachable state $x\in X$, there exist $t\in T$ and $x'\in X$ such that $(x,t,x')\in\dt$;
		\item\label{item12_Det_PN} 	
			$\Scal$ is \emph{prompt} or \emph{divergence-free}, 
			i.e., for every reachable state $x\in X$ and every
			nonempty unobservable event sequence $s\in(T_{uo})^+$, there exists no transition
			sequence $x\xrightarrow[]{s}x$ in $\Scal$.
	\end{enumerate}
\end{assumption}

One sees \eqref{item11_Det_PN} implies $L^{\omega}(\Scal)\ne\emptyset$ if $X_0\ne\emptyset$;
while \eqref{item12_Det_PN} implies for all $s\in
L^{\omega}(\Scal)$, $\ell(s)\in\Sig^{\omega}$; hence \eqref{item11_Det_PN} and \eqref{item12_Det_PN}
together imply $\LM^{\omega}(\Scal)\ne\emptyset$ if $X_0\ne\emptyset$, but not vice versa.


%

\section{Main results}\label{sec3:mainresult}

\subsection{Preliminary results}

The definitions of strong detectability, strong periodic detectability,
and strong periodic D-detectability for FSAs are as follows \cite{Shu2011GDetectabilityDES}.

\begin{definition}[SD]\label{def4_Det_PN}
	An FSA $\Scal$ as in \eqref{DES_Det_NDES} is
	called \emph{strongly detectable} if there exists a positive
	integer $k$ such that for each infinite-length event sequence $s\in L^{\omega}({\Scal})$
	generated by $\Scal$, for each prefix $s'\sqsubset s$, if $|\ell(s')|>k$ then
	$|\Mt({\Scal},\ell(s'))|=1$. 
\end{definition}

\begin{definition}[SPD]\label{def9_Det_PN}
	An FSA $\Scal$ is
	called \emph{strongly periodically detectable} if there exists a positive
	integer $k$ such that for each $s\in L^{\omega}({\Scal})$
	and each $s'\sqsubset s$, there is $s''\in T^*$ such that
	$|\ell(s'')|<k$, $s's''\sqsubset s$, and $|\Mt(\Scal,\ell(s's''))|=1$.
\end{definition}

In order to formulate strong periodic D-detectability, we specify a set 
$$T_{\spec}\subset X\times X$$
of crucial state pairs that should be separated.
\begin{definition}[$T_{\spec}$-SPDD]\label{def11_Det_PN}
	An FSA $\Scal$ is
	called \emph{strongly periodically D-detectable with respect to $T_{\spec}$}
	if there exists a positive
	integer $k$ such that for each $s\in L^{\omega}({\Scal})$
	and each $s'\sqsubset s$, there is $s''\in T^*$ such that
	$|\ell(s'')|<k$, $s's''\sqsubset s$, and $(\Mt(\Scal,\ell(s's''))\times\Mt(\Scal,\ell(s's'')))
	\cap T_{\spec}=\emptyset$.
\end{definition}

In order to verify detectability of an FSA $\Scal$,
an \emph{observer} 
\begin{equation}\label{observer_DES}
	\Scal_{obs}:=(2^X\setminus\{\emptyset\},\Sig,\Mt(\Scal,\ep),\dt_{obs})
\end{equation}
as a deterministic FSA was constructed in \cite{Shu2007Detectability_DES}, where
$\Mt(\Scal,\ep)$ is the unique initial state; for all $X'\in 2^X$ and $\s\in\Sig^*$,
$\dt_{obs}(\Mt(\Scal,\ep),\s)=X'$ if and only if $X'=\Mt(\Scal,\s)$.
The size of $\Scal_{obs}$ is exponential of that of $\Scal$.

Later in \cite{Shu2011GDetectabilityDES}, a \emph{detector}
\begin{equation}\label{detector_DES}
	\Scal_{det}:=(Q,\Sig,\Mt(\Scal,\ep),\dt_{det})
\end{equation}
that is a nondeterministic FSA was used to provide polynomial-time algorithms
for verifying strong detectability and strong periodic detectability under Assumption~\ref{assum1_Det_PN},
where $Q\subset 2^{X}\setminus\{\emptyset\}$ consists of $\Mt(\Scal,\ep)$
and subsets of $X$ with cardinality $\le 2$; for all $q,q'\in Q$, and $\s\in\Sig$,
$(q,\s,q')\in\dt_{det}$ if and only if either (1) $|(\UR\circ\Reach_\s)(q)|>1$,
$q'\subset(\UR\circ\Reach_\s)(q)$, and $|q'|=2$,
or (2) $|(\UR\circ\Reach_\s)(q)|=1$ and $q'=(\UR\circ\Reach_\s)(q)$.
The size of $\Scal_{det}$ is polynomial of that of $\Scal$.
The results obtained in \cite{Shu2011GDetectabilityDES} are as follows.
\begin{proposition}[\cite{Shu2011GDetectabilityDES}]\label{prop1_spdet_FSA}
	Consider an FSA $\Scal$. Under Assumption~\ref{assum1_Det_PN}, $\Scal$
	is strongly detectable if and only if in $\Scal_{det}$, any state reachable
	from any reachable transition cycle is a singleton; $\Scal$ is strongly periodically detectable 
	if and only if in $\Scal_{det}$, every reachable transition cycle contains at least one
	singleton; $\Scal$ is strongly periodically D-detectable if and only if in $\Scal_{obs}$,
	every reachable transition cycle contains at least one state $q$ such that $(q\times q)\cap 
	T_{\spec}=\emptyset$.
\end{proposition}

In \cite{Zhang2019KDelayStrDetDES}, in order to verify (delayed) strong detectability
of $\Scal$, the \emph{self-composition}
\begin{equation}\label{ConCom_Det_PN}
	\CCa(\Scal)=(X',T',X_0',\dt')
\end{equation} of $\Scal$ (i.e., the concurrent composition of $\Scal$ and itself) was constructed as follows:
\begin{itemize}
	\item $X'=X\times X$;
	\item $T'=T_o'\cup T'_{uo}$, where $T_o'=\{(\breve{t},\breve{t}')|\breve{t},\breve{t}'\in T,
		\ell(\breve{t})=\ell(\breve{t}')\in\Sig\}$,
		$T'_{uo}=\{(\breve{t},\epsilon)|\breve{t}\in T,\ell(\breve{t})=\epsilon\}\cup
		\{(\epsilon,\breve{t})|\breve{t}\in T,\ell(\breve{t})=\epsilon\}$;
	\item $X_0'=X_0\times X_0$;
	\item for all $(\breve{x}_1,\breve{x}_1'),(\breve{x}_2,\breve{x}_2')\in X'$, $(\breve{t},\breve{t}')
		\in T_o'$, $(\breve{t}'',\epsilon)\in T'_{uo}$,
		and $(\epsilon,\breve{t}''')\in T'_{uo}$,
		\begin{itemize}
			\item $((\breve{x}_1,\breve{x}_1'),(\breve{t},\breve{t}'),(\breve{x}_2,\breve{x}_2'))\in\dt'$ 
				if and only if $(\breve{x}_1,\breve{t},\breve{x}_2),\\(\breve{x}_1',\breve{t}',\breve{x}_2')\in\dt$,
			\item $((\breve{x}_1,\breve{x}_1'),(\breve{t}'',\epsilon),(\breve{x}_2,\breve{x}_2'))\in\dt'$ 
				if and only if $(\breve{x}_1,\breve{t}'',\breve{x}_2)\in\dt$, $\breve{x}_1'=\breve{x}_2'$,
			\item $((\breve{x}_1,\breve{x}_1'),(\epsilon,\breve{t}'''),(\breve{x}_2,\breve{x}_2'))\in\dt'$ 
				if and only if $\breve{x}_1=\breve{x}_2$, $(\breve{x}_1',\breve{t}''',\breve{x}_2')\in\dt$.
		\end{itemize}
	\end{itemize}

For an event sequence $s'\in (T')^{*}$, $s'(L)$ and $s'(R)$ denote its left and right
components, respectively. Similarly for $x'\in X'$, denote $x'=:(x'(L),x'(R))$.
In addition, for every $s'\in(T')^{*}$, $\ell(s')$
denotes $\ell(s'(L))$ or $\ell(s'(R))$, since $\ell(s'(L))=\ell(s'(R))$. 
In the above construction, $\CCa(\Scal)$ aggregates every pair of transition sequences of 
$\mathcal{S}$ producing the same label sequence. The size of $\CCa(\Scal)$ is polynomial
of that of $\Scal$. 

\subsection{Verifying strong periodic detectability}

In order to verify strong periodic detectability without any assumption,
we first characterize its negation.
By directly observing Definition~\ref{def9_Det_PN}, the following result follows.
\begin{proposition}\label{prop3_spdet_FSA}
	An FSA $\Scal$ is
	not strongly periodically detectable if and only if for every positive
	integer $k$, there exists $s_k\in L^{\omega}({\Scal})$ and
	prefix $s'\sqsubset s_k$ such that for all $s''\in T^*$, $s's''\sqsubset s_k$
	and $|\ell(s'')|<k$ imply $|\Mt(\Scal,\ell(s's''))|>1$.
\end{proposition}

By Proposition~\ref{prop3_spdet_FSA}, the following proposition holds.
\begin{proposition}\label{prop4_spdet_FSA}
	An FSA $\Scal$ is
	not strongly periodically detectable if and only if at least one of the following two
	conditions holds.
	\begin{myenumerate}
		\item\label{item1_spdet_FSA}
			There exists $\gamma\in\LM(\Scal)$ and $x\in\Mt(\Scal,\gamma)$ such that $|\Mt(\Scal,\gamma)|>1$
			and there is a transition sequence $x\xrightarrow[]{s_1}x'\xrightarrow[]{s_2}x'$ for some
			$s_1\in (T_{uo})^*$, $s_2\in(T_{uo})^+$, $x'\in X$.
		\item\label{item2_spdet_FSA}
			There exists $\alpha\beta\in\LM(\Scal)$ such that $|\beta|>0$,
			$\Mt(\Scal,\alpha)=\Mt(\Scal,\alpha\beta)$, and $|\Mt(\Scal,\alpha
			\bar\beta)|>1$ for all $\bar\beta\sqsubset\beta$.
	\end{myenumerate}
	
\end{proposition}

\begin{proof}
	``if'': Assume \eqref{item1_spdet_FSA} holds. Then there exists a transition sequence
	$x_0\xrightarrow[]{s_{\gamma}}x\xrightarrow[]{s_1}x'\xrightarrow[]{s_2}x'$ 
	such that $x_0\in X_0$ and $\ell(s_{\gamma})=\gamma$. For every positive 
	integer $k$, choose $s_k=s_{\gamma}s_1(s_2)^{\omega}\in L^{\omega}(\Scal)$,
	then for every $s''\sqsubset s_1(s_2)^{\omega}$, one has $\ell(s'')=\ep$
	and $|\Mt(\Scal,\ell(s_{\gamma}
	s''))|=|\Mt(\Scal,\gamma)|>1$, which violates strong periodic detectability
	by Proposition~\ref{prop3_spdet_FSA}.

	Assume \eqref{item2_spdet_FSA} holds. Then $\alpha\beta^{\omega}\in\LM^{\omega}(\Scal)$.
	For every positive integer $k$, choose $s_k=s_{\alpha}s_{\beta}\in L^{\omega}(\Scal)$
	such that $\ell(s_\alpha)=\alpha$ and $\ell(s_\beta)=\beta^{\omega}$. Then for every $s_{\beta}'
	\sqsubset s_{\beta}$, one has $|\Mt(\Scal,\ell(s_\alpha s_{\beta}'))|>1$,
	which also violates strong periodic detectability by Proposition~\ref{prop3_spdet_FSA}.

	``only if'':  Assume $\Scal$ is not strongly periodically detectable and
	\eqref{item2_spdet_FSA} does not hold, next we prove \eqref{item1_spdet_FSA}
	holds.

	Since $\Scal$ is not strongly periodically detectable, by Proposition~\ref{prop3_spdet_FSA},
	choose integer $k>|2^X|$, $s_k\in L^{\omega}({\Scal})$, and
	prefix $s'\sqsubset s_k$ such that for all $s''\in T^*$, $s's''\sqsubset s_k$
	and $|\ell(s'')|<k$ imply $|\Mt(\Scal,\ell(s's''))|>1$. Since \eqref{item2_spdet_FSA}
	does not hold, one has $\ell(s_k)\in\Sig^*$ and $|\ell(s_k)|<k+|\ell(s')|$.
	Otherwise if $|\ell(s_k)|\ge k+|\ell(s')|$ or $\ell(s_k)\in\Sig^{\omega}$,
	we can choose $\bar s''$ such that $s'\bar s''\sqsubset s_k$ and $|\ell(\bar s'')|=
	k$, then by the Pigeonhole Principle and $k>|2^X|$,
	there exist $\bar s_1'',\bar s_2''\sqsubset
	\bar s''$ such that $|\ell(\bar s_1'')|<|\ell(\bar s_2'')|$ and $\Mt(\Scal,\ell(s'\bar
	s_1''))=\Mt(\Scal,\ell(s'\bar s_2''))$, that is, \eqref{item2_spdet_FSA} holds.
	Then $s_k=s'\hat s_1''\hat s_2''$, where $\hat s_1''\in T^*$, $\hat s_2''\in (T_{uo})^
	{\omega}$. Moreover, one has $|\Mt(\Scal,\ell(s'
	\hat s_1''))|>1$, and also by the Pigeonhole Principle
	there exists a transition sequence $x_0\xrightarrow[]{s'
	\hat s_1''}x\xrightarrow[]{\tilde s_1''}x'\xrightarrow[]{\tilde s_2''}x'$ for some
	$x_0\in X_0$, $x,x'\in X$, $\tilde s_1''\in (T_{uo})^*$, and $\tilde s_2''\in
	(T_{uo})^+$, i.e., \eqref{item1_spdet_FSA} holds.
\end{proof}

\begin{theorem}\label{thm3_spdet_FSA}
	An FSA $\Scal$ is
	not strongly periodically detectable if and only if in its observer $\Scal_{obs}$
	as in \eqref{observer_DES}, at least one of the two following conditions holds.
	\begin{myenumerate}
		\item\label{item3_spdet_FSA}
			There is a reachable state $q\in 2^X$ in $\Scal_{obs}$ and $x\in q$ such that $|q|>1$ and
			there is a transition sequence $x\xrightarrow[]{s_1}x'\xrightarrow[]{s_2}x'$ in $\Scal$ for 
			some $s_1\in (T_{uo})^*$, $s_2\in(T_{uo})^+$, $x'\in X$.
		\item\label{item4_spdet_FSA}
			There is a reachable transition cycle such that
			no state in the cycle is a singleton.

	\end{myenumerate}
\end{theorem}

\begin{proof}
	By definition, one sees that for all $\sigma\in\LM(\Scal)$,
	$\Mt(\Scal,\sigma)=\dt_{obs}(\Mt(\Scal,\ep),\sigma)$. Then \eqref{item3_spdet_FSA} 
	(resp., \eqref{item4_spdet_FSA}) of this theorem is equivalent to \eqref{item1_spdet_FSA}
	(resp., \eqref{item2_spdet_FSA}) of Proposition~\ref{prop4_spdet_FSA}.
\end{proof}

Theorem~\ref{thm3_spdet_FSA} provides an exponential-time algorithm for verifying
strong periodic detectability of $\Scal$. Next we obtain a polynomial-time
verification algorithm by simplifying Theorem~\ref{thm3_spdet_FSA}.
To this end, we need to prove a relationship between $\Scal_{det}$ and $\Scal_{obs}$.

\begin{proposition}\label{prop5_spdet_FSA}
	Consider an FSA $\Scal$. For every transition $(q,\s,q')\in\dt_{obs}$,
	for every $\emptyset\ne\bar q'\subset q'$ satisfying $|\bar q'|=2$ if $|q'|\ge 2$,
	there is $\bar q\subset q$ such that 
	$(\bar q,\s,\bar q')\in\dt_{det}$, 
	where $|\bar q|=2$ if $|q|\ge2$.
\end{proposition}

\begin{proof}
	We only need to prove the case $|q|\ge 2$ and $|q'|\ge 2$, the other cases hold similarly.
	Arbitrarily choose $\{x_1,x_2\}=\bar q'\subset q'$ such that $x_1\ne x_2$. By definition,
	either (1) there exists $x_3\in X$, $t_1,t_2\in T_o$, $s_1,s_2\in (T_{uo})^*$ such that
	$(x_3,t_1s_1,x_1),(x_3,t_2s_2,x_2)\in\dt$ and $\ell(t_1)=\ell(t_2)=\s$, or (2)
	there exist $x_4,x_5\in X$, $t_1,t_2\in T_o$, $s_1,s_2\in (T_{uo})^*$ such that
	$x_4\ne x_5$, $(x_4,t_1s_1,x_1),(x_5,t_2s_2,x_2)\in\dt$ and $\ell(t_1)=\ell(t_2)=\s$.
	If (1) holds, we choose $\bar q=\{x_3,x_6\}$, where $x_6\in q\setminus\{x_3\}$;
	if (2) holds, we choose $\bar q=\{x_4,x_5\}$. By definition, no matter (1) or (2)
	holds, one has $(\bar q,\s,\bar q')\in\dt_{det}$.
\end{proof}

\begin{theorem}\label{thm4_spdet_FSA}
	An FSA $\Scal$ is
	not strongly periodically detectable if and only if in its detector $\Scal_{det}$
	as in \eqref{detector_DES}, at least one of the two following conditions holds.
	\begin{myenumerate}
		\item\label{item5_spdet_FSA}
			There is a reachable state $q'\in Q$ and $x\in q'$ such that $|q'|>1$
			and there is a transition
			sequence $x\xrightarrow[]{s_1}x'\xrightarrow[]{s_2}x'$ in $\Scal$ for some
			$s_1\in (T_{uo})^*$, $s_2\in(T_{uo})^+$, $x'\in X$.
		\item\label{item6_spdet_FSA}
			There is a reachable transition cycle such that
			all states in the cycle have cardinality $2$.
	\end{myenumerate}
\end{theorem}

\begin{proof}
	We use Theorem~\ref{thm3_spdet_FSA} and Proposition~\ref{prop5_spdet_FSA} to prove this 
	result.

	We first check \eqref{item5_spdet_FSA} is equivalent to \eqref{item3_spdet_FSA}.

	``$\Rightarrow$'': Assume \eqref{item5_spdet_FSA} holds.
	In $\Scal_{det}$, choose a transition sequence $\Mt(\Scal,\ep)\xrightarrow[]
	{\alpha}q'$. One then has $q'\subset\Mt(\Scal,\alpha)=\dt_{obs}(\Mt(\Scal,\ep),\alpha)$,
	hence \eqref{item3_spdet_FSA} of Theorem~\ref{thm3_spdet_FSA} holds.

	``$\Leftarrow$'': Assume \eqref{item3_spdet_FSA} holds.
	In $\Scal_{obs}$, choose a transition sequence $\Mt(\Scal,\ep)\xrightarrow[]
	{\alpha}q$. By Proposition~\ref{prop5_spdet_FSA}, moving backward on $\Mt(\Scal,\ep)\xrightarrow[]
	{\alpha}q$ from $q$ to $\Mt(\Scal,\ep)$, we obtain a transition sequence
	$\Mt(\Scal,\ep)\xrightarrow[]{\alpha}q'$ of $\Scal_{det}$ such that $x\in q'\subset q$ and $|q'|>1$,
	i.e., \eqref{item5_spdet_FSA} holds.

	We secondly check \eqref{item6_spdet_FSA} is equivalent to \eqref{item4_spdet_FSA}.

	``$\Rightarrow$'': Assume \eqref{item6_spdet_FSA} holds.
	In $\Scal_{det}$, choose a transition sequence $\Mt(\Scal,\ep)\xrightarrow[]{\alpha}
	q\xrightarrow[]{\beta}q$ such that in $q\xrightarrow[]{\beta}q$
	all states are of cardinality $2$ and $|\beta|>0$. Without loss of generality, we assume 
	$|\beta|>|2^{X}|$, because otherwise we can repeat $q\xrightarrow[]{\beta}q$ for
	$|2^{X}|+1$ times. By definition, one has for all
	$\beta'\sqsubset\beta$, $|\Mt(\Scal,\alpha\beta')|>1$. Then by the Pigeonhole Principle,
	there exist $\beta_1,\beta_2\sqsubset\beta$ such that $|\beta_1|<|\beta_2|$
	and $\Mt(\Scal,\alpha\beta_1)=\Mt(\Scal,\alpha\beta_2)$. Then in observer $\Scal_{obs}$,
	one has $\dt_{obs}(\Mt(\Scal,\ep),\alpha\beta_1)=\Mt(\Scal,\alpha\beta_1)=\Mt(\Scal,\alpha\beta_2)=
	\dt_{obs}(\Mt(\Scal,\ep),\alpha\beta_2)$, and for every $\beta'\sqsubset\beta$, 
	$\dt_{obs}(\Mt(\Scal,\ep),\alpha\beta')=\Mt(\Scal,\alpha\beta')$ has cardinality $>1$.
	Thus, \eqref{item4_spdet_FSA} of Theorem~\ref{thm3_spdet_FSA} holds.

	``$\Leftarrow$'': Assume \eqref{item4_spdet_FSA} holds.
	In $\Scal_{obs}$, choose a transition sequence $\Mt(\Scal,\ep)\xrightarrow[]{\alpha}
	q_1\xrightarrow[]{\beta_1}\cdots\xrightarrow[]{\beta_n}q_{n+1}$ such that $n\ge|X|^2$, 
	$q_1=q_{n+1}$, $|q_1|,\dots,|q_{n+1}|>1$, and $\beta_1,\dots,\beta_n\in\Sig$.
	By Proposition~\ref{prop5_spdet_FSA} from $n+1$ to $2$,
	we obtain $q_i'\subset q_i$ for all $i\in\llb 1,n+1\rrb$
	such that $|q_1'|=\cdots=|q_{n+1}'|=2$ and a transition sequence
	$q_1'\xrightarrow[]{\beta_1}\cdots\xrightarrow[]{\beta_n}q_{n+1}'$ of $\Scal_{det}$.
	Moreover, by Proposition~\ref{prop5_spdet_FSA}, we obtain a transition sequence 
	$\Mt(\Scal,\ep)\xrightarrow[]{\alpha}q_1'$ of $\Scal_{det}$. By the Pigeonhole Principle,
	\eqref{item6_spdet_FSA} holds.
\end{proof}

In order to check condition \eqref{item6_spdet_FSA}, one could firstly use Tarjan algorithm to
compute all reachable strongly connected components of $\Scal_{det}$, which takes time linear 
in the size of $\Scal_{det}$; secondly at each component, remove all singletons
and then check whether there is a cycle. If and only if in some reachable component, such a cycle
exists, \eqref{item6_spdet_FSA} holds. Hence Theorem~\ref{thm4_spdet_FSA} provides a polynomial-time
algorithm for verifying strong periodic detectability. 
Moreover, Theorem~\ref{thm4_spdet_FSA} also implies an $\NL$ upper bound for strong periodic detectability. 

\begin{theorem}\label{thm5_spdet_FSA}
	The problem of verifying strong periodic detectability of FSA $\Scal$ belongs to $\NL$.
\end{theorem}

\begin{proof}
	We only need to prove that both \eqref{item5_spdet_FSA} and \eqref{item6_spdet_FSA} of
	Theorem~\ref{thm4_spdet_FSA} can be verified in $\NL$. Then by $\NL=\coNL$, this theorem holds.
	We do not need to compute the whole $\Scal_{det}$.

	For \eqref{item5_spdet_FSA}: Guess states $x,x',x''\in X$, check (\romannumeral1)
	$x\ne x''$, (\romannumeral2) $x,x''\in
	\Mt(\Scal,\ep)$ or $\{x,x''\}$ is reachable in $\Scal_{det}$, (\romannumeral3) $x'$ is equal to $x$ or 
	there is an unobservable transition sequence from $x$ to $x'$ in $\Scal$, 
	and (\romannumeral4) there is an unobservable transition cycle from $x'$ to itself, all by nondeterministic 
	search.

	For \eqref{item6_spdet_FSA}: Guess different states $x,\bar x\in X$, check (\romannumeral1)
	$\{x,\bar x\}$ is reachable in $\Scal_{det}$, (\romannumeral2)
	$\{x,\bar x\}$ belongs to a transition cycle whose 
	states all have cardinality $2$.
\end{proof}

\begin{example}\label{exam1_spdet_FSA}
	We give two examples to illustrate Theorem~\ref{thm4_spdet_FSA}.
	Consider two FSAs $\Scal_1$ and $\Scal_2$ shown in Fig.~\ref{fig1_spdet_FSA}.
	One sees that $\Scal_1$ satisfies Assumption~\ref{assum1_Det_PN}. However, $\Scal_2$ does not
	satisfy Assumption~\ref{assum1_Det_PN}, as $x_1$ is a deadlock (violating \eqref{item11_Det_PN}
	of Assumption~\ref{assum1_Det_PN}), and there is a
	reachable unobservable transition cycle $x_2\xrightarrow[]{t_4}x_2$ (violating \eqref{item12_Det_PN}
	of Assumption~\ref{assum1_Det_PN}).
	\begin{figure}[!htpb]
			\tikzset{global scale/.style={
    scale=#1,
    every node/.append style={scale=#1}}}
		\begin{center}
			\begin{tikzpicture}[global scale = 1.0,
				>=stealth',shorten >=1pt,thick,auto,node distance=2.5 cm, scale = 0.8, transform shape,
	->,>=stealth,inner sep=2pt,
				every transition/.style={draw=red,fill=red,minimum width=1mm,minimum height=3.5mm},
				every place/.style={draw=blue,fill=blue!20,minimum size=7mm}]
				\tikzstyle{emptynode}=[inner sep=0,outer sep=0]
				\node[state, initial, initial where = right] (x0) {$x_0$};
				\node[state] (x1) [left of = x0] {$x_1$};
				\node[state] (x2) [below of = x0] {$x_2$};

				\path[->]
				(x0) edge node [above, sloped] {$t_1(a)$} (x1)
				(x0) edge node [above, sloped] {$t_2(a)$} (x2)
				(x1) edge [loop below] node [below, sloped] {$t_3(a)$} (x1)
				(x2) edge [loop right] node [above, sloped] {$t_4(a)$} (x2)
				;

				\node[state] (x1') [right of = x0] {$x_1$};
				\node[state, initial, initial where = right] (x0') [right of = x1'] {$x_0$};
				\node[state] (x2') [below of = x0'] {$x_2$};

				\path[->]
				(x0') edge node [above, sloped] {$t_1(a)$} (x1')
				(x0') edge node [above, sloped] {$t_2(a)$} (x2')
				(x2') edge [loop right] node [above, sloped] {$t_4(\ep)$} (x2')
				;

			\end{tikzpicture}
	\end{center}
	\caption{FSA $\Scal_1$ (left) and FSA $\Scal_2$ (right), where a state with an input arrow from nowhere 
	is initial (e.g., $x_0$), the letters beside arrows outside $()$ denote events, the letters in $()$ denote
	the corresponding labels/outputs.} 
	\label{fig1_spdet_FSA}
	\end{figure}
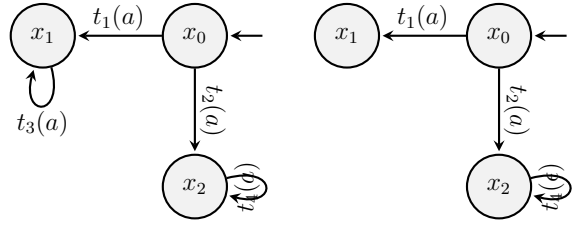

	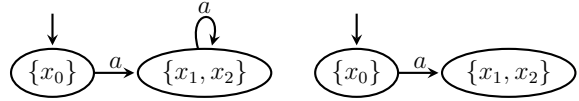
\begin{figure}[!htpb]
			\tikzset{global scale/.style={
    scale=#1,
    every node/.append style={scale=#1}}}
		\begin{center}
			\begin{tikzpicture}[global scale = 1.0,
				>=stealth',shorten >=1pt,thick,auto,node distance=2.5 cm, scale = 0.8, transform shape,
	->,>=stealth,inner sep=2pt,
				every transition/.style={draw=red,fill=red,minimum width=1mm,minimum height=3.5mm},
				every place/.style={draw=blue,fill=blue!20,minimum size=7mm}]
				\tikzstyle{emptynode}=[inner sep=0,outer sep=0]
				\node[elliptic state, initial, initial where = above] (x0) {$\{x_0\}$};
				\node[elliptic state] (x2) [right of = x0] {$\{x_1,x_2\}$};

				\path[->]
				(x0) edge node [above, sloped] {$a$} (x2)
				(x2) edge [loop above] node {$a$} (x2)
				;

				\node[elliptic state, initial, initial where = above] (x0') [right of = x2] {$\{x_0\}$};
				\node[elliptic state] (x2') [right of = x0'] {$\{x_1,x_2\}$};

				\path[->]
				(x0') edge node [above, sloped] {$a$} (x2')
				;

			\end{tikzpicture}
	\end{center}
	\caption{Detectors $\Scal_{1det}$ (left, the same as observer $\Scal_{1obs}$)
	and $\Scal_{2det}$ (right, the same as observer $\Scal_{2obs}$)
	of FSA $\Scal_1$ and FSA $\Scal_2$ shown in Fig.~\ref{fig1_spdet_FSA}.}
	\label{fig2_spdet_FSA}
	\end{figure}

	Their detectors $\Scal_{1det}$ and $\Scal_{2det}$ are shown in Fig.~\ref{fig2_spdet_FSA}.
	One sees $\Scal_{1det}$ satisfies \eqref{item6_spdet_FSA} of Theorem~\ref{thm4_spdet_FSA}
	because there is a self-loop on reachable state $\{x_1,x_2\}$,
	but does not satisfy \eqref{item5_spdet_FSA} because
	$\{x_1,x_2\}$ is the unique reachable state of cardinality $2$ and
	in $\Scal_1$
	there is no infinitely long unobservable transition sequence starting at $x_1$, the same for $x_2$.
	$\Scal_{2det}$ satisfies \eqref{item5_spdet_FSA} because $\{x_1,x_2\}$ is reachable 
	in $\Scal_{2det}$ and in $\Scal_2$, starting at $x_2$ there is an infinite-length unobservable transition sequence,
	but does not satisfy \eqref{item6_spdet_FSA} because there is no cycle all of whose states are of
	cardinality $2$. Hence by Theorem~\ref{thm4_spdet_FSA},
	neither $\Scal_1$ nor $\Scal_2$ is strongly periodically detectable.
\end{example}

\begin{remark}\label{rem1_spdet_FSA}
	By Example~\ref{exam1_spdet_FSA}, one sees that \eqref{item5_spdet_FSA} and 
	\eqref{item6_spdet_FSA} do not imply each other.
	So they cannot take the place of each other when verifying strong periodic
	detectability. 
	Let us compare Theorem~\ref{thm4_spdet_FSA} with Proposition~\ref{prop1_spdet_FSA}.
	One directly sees that the equivalent condition for strong periodic detectability
	under Assumption~\ref{assum1_Det_PN} shown in Proposition~\ref{prop1_spdet_FSA}
	is exactly negation of \eqref{item6_spdet_FSA}.
	By Proposition~\ref{prop1_spdet_FSA}, $\Scal_2$ is strongly periodically detectable vacuously.
	Then Proposition~\ref{prop1_spdet_FSA} does not
	always work correctly if Assumption~\ref{assum1_Det_PN} is not satisfied.
\end{remark}

Next we show that a slight variant of the concurrent-composition structure can also provide
an $\NL$ upper bound for strong periodic detectability. The concurrent-composition
structure has essentially different features compared with the detector structure. On the 
one hand, a detector
tracks output sequences and collects all states between only unobservable transitions and divides
them into subsets of cardinality $2$. So a detector does not reflect information in
unobservable transitions. However, the concurrent-composition structure can do that.
On the other hand, a concurrent composition collects all pairs of transition sequences generating the same
output sequence, but sometimes does not collect different transitions starting at
the same state. However, a detector can do that. For example, consider states $x_1,x_2,x_3,x_4$ 
such that $x_1\ne x_2$ and $x_3\ne x_4$, there exist transitions $x_1\xrightarrow[]{t_1}x_3$,
$x_1\xrightarrow[]{t_2}x_4$
satisfying $\ell(t_1)=\ell(t_2)\ne\ep$, but there is no transition $x_2\xrightarrow[]{t_3}x_4$
satisfying $\ell(t_3)=\ell(t_1)$. Then in $\Scal_{det}$ there is a transition $\{x_1,x_2\}
\xrightarrow[]{\ell(t_1)}\{x_3,x_4\}$, but in $\CCa(\Scal)$ there is no transition $(x_1,x_2)
\xrightarrow[]{t'}(x_3,x_4)$ for any $t'\in T'$ satisfying $\ell(t')=\ell(t_1)$.
Next we add additional transitions into $\CCa(\Scal)$ to remove this drawback of $\CCa(\Scal)$
so that a verification algorithm for strong periodic detectability could be derived.

Consider an FSA $\Scal$ as in \eqref{DES_Det_NDES} and its self-composition
$\CCa(\Scal)$ as in \eqref{ConCom_Det_PN}. We construct a variant
\begin{equation}\label{ConCom_Rever_Det_PN}
	\CCaLeft(\Scal)=(X',T'\cup\{\varepsilon\},X_0',\dt'_{\leftarrow_\ep})
\end{equation}
from $\CCa(\Scal)$ as follows: For all $x_1,x_2,x_3,x_4\in X$, and $t'\in T'$ such that
$x_1\ne x_2$, $((x_1,x_1),t',(x_3,x_4))\in\dt'$ (resp., $((x_2,x_2),t',(x_3,x_4))\in\dt'$),
but $((x_1,x_2),\bar t',(x_3,x_4))\notin\dt'$ for any $\bar t'\in T'$,
add transition $((x_1,x_2),\varepsilon,(x_1,x_1))$ (resp., $((x_1,x_2),\varepsilon,(x_2,x_2))$),
where we let $\ell(\varepsilon)=\ep$. We call $\CCaLeft(\Scal)$ \emph{$\varepsilon$-extended
self-composition} of $\Scal$.

One can see the following proposition.
\begin{proposition}\label{prop8_spdet_FSA}
	Consider an FSA $\Scal$ as in \eqref{DES_Det_NDES}, its observer $\Scal_{obs}$
	as in \eqref{observer_DES}, its detector $\Scal_{det}$
	as in \eqref{detector_DES}, and its $\varepsilon$-extended self-composition
	$\CCaLeft(\Scal)$ as in \eqref{ConCom_Rever_Det_PN}. Assume states $x_1,x_2,x_3,x_4\in X$
	such that $x_1\ne x_2$ and $x_3\ne x_4$. The following hold.
	\begin{myenumerate}
		\item\label{item11_spdet_FSA}
			For every transition $\{x_1,x_2\}\xrightarrow[]{\s}\{x_3,x_4\}$ in $\Scal_{det}$,
			there is an observable transition sequence $(x_1,x_2)\xrightarrow[]{s'}(x_3,x_4)$
			or $(x_1,x_2)\xrightarrow[]{s'}(x_4,x_3)$ in $\CCaLeft(\Scal)$ such that $\ell(s')=\s$.
		\item\label{item12_spdet_FSA}
			For every transition $\{x_1,x_2\}\xrightarrow[]{\s}\{x_3\}$ in $\Scal_{det}$,
			there is an observable transition sequence $(x_1,x_2)\xrightarrow[]{s'}(x_3,x_3)$
			in $\CCaLeft(\Scal)$ such that $\ell(s')=\s$.
		\item\label{item13_spdet_FSA}
			For every transition $\{x_1\}\xrightarrow[]{\s}\{x_3,x_4\}$ in $\Scal_{det}$,
			there is an observable transition sequence $(x_1,x_1)\xrightarrow[]{s'}(x_3,x_4)$
			in $\CCaLeft(\Scal)$ such that $\ell(s')=\s$.
		\item\label{item14_spdet_FSA}
			For every transition $\{x_1\}\xrightarrow[]{\s}\{x_3\}$ in $\Scal_{det}$,
			there is an observable transition sequence $(x_1,x_1)\xrightarrow[]{s'}(x_3,x_3)$
			in $\CCaLeft(\Scal)$ such that $\ell(s')=\s$.
		\item\label{item15_spdet_FSA}
			In $\CCaLeft(\Scal)$, consider an arbitrary transition sequence $x_0'\xrightarrow[]{s_0'}
			x_1'\xrightarrow[]{t_1'}x_2'\xrightarrow[]{s_1'}\cdots\xrightarrow[]{t_n'}x_{2n}'
			\xrightarrow[]{s_n'}x_{2n+1}'$, where $x_0'\in X_0'$, $x_1',\dots,x_{2n+1}'\in X'$,
			$s_0',\dots,s_n'\in(T_{uo}'\cup\{\varepsilon\})^*$, $t_1',\dots,t_n'\in T_o'$.
			For every $i\in\llb 0,n\rrb$, denote the union of all states of unobservable transition
			sequence $x_{2i}\xrightarrow[]{s_i'}x_{2i+1}$ by $q_i$, then we obtain a sequence
			$q_0\xrightarrow[]{\ell(t_1')}\cdots \xrightarrow[]{\ell(t_n')}q_n$. Then for every
			$i\in\llb 1,n\rrb$, there exists $\bar q_i\supset q_i$ such that $\Mt(\Scal,\ep)\xrightarrow[]{\ell(t_1')}
			\bar q_1\xrightarrow[]{\ell(t_2')}\cdots \xrightarrow[]{\ell(t_n')}\bar q_n$
			is a transition sequence of $\Scal_{obs}$. 
	\end{myenumerate}
\end{proposition}

\begin{proof}
	\eqref{item11_spdet_FSA} We need to consider four different cases of transition
	sequences in $\Scal$ (shown in Figs.~\ref{fig:Case1&2}, \ref{fig:Case3&4})
	that form the transition $\{x_1,x_2\}\xrightarrow[]{\s}\{x_3,x_4\}$ in $\Scal_{det}$,
	where in these figures, $t_1,t_2\in T_o$, $\ell(t_1)=\ell(t_2)=\sigma$, $s_1,s_2,s_3,s_4\in (T_{uo})^{*}$,
	$x_5,x_6,x_7,x_8\in X$.
	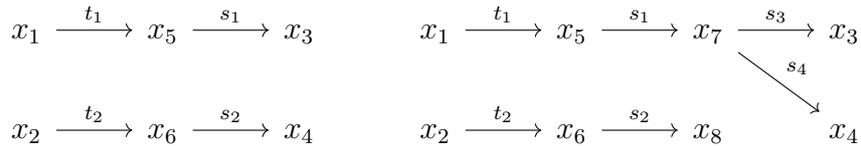
\begin{figure}[!htbp]
		\begin{center}
			\begin{tikzcd}
				x_1 \arrow{r}{t_1} & x_5 \arrow{r}{s_1} & x_3 &
				x_1 \arrow{r}{t_1} & x_5 \arrow{r}{s_1} & x_7 \arrow{r}{s_3} \arrow{rd}{s_4} & x_3\\
				x_2 \arrow{r}{t_2} & x_6 \arrow{r}{s_2} & x_4 &
				x_2 \arrow{r}{t_2} & x_6 \arrow{r}{s_2} & x_8 & x_4
			\end{tikzcd}
		\end{center}
		\caption{Case 1 (left). Case 2 (right).}
		\label{fig:Case1&2}
	\end{figure}
	\begin{figure}[!htbp]
		\begin{center}
			\begin{tikzcd}
				x_1 \arrow{r}{t_1} \arrow{rd}{t_2} & x_5 \arrow{r}{s_1} & x_3 &
				x_1 \arrow{r}{t_1} \arrow{rd}{t_2} & x_5 \arrow{r}{s_1} & x_7 \arrow{r}{s_3} \arrow{rd}{s_4} & x_3\\
				x_2 & x_6 \arrow{r}{s_2} & x_4 &
				x_2  & x_6 \arrow{r}{s_2} & x_8 & x_4
			\end{tikzcd}
		\end{center}
		\caption{Case 3 (left). Case 4 (right).}
		\label{fig:Case3&4}
	\end{figure}
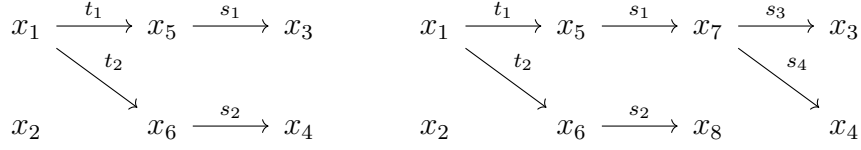

	We need to prove for each case, there is an observable transition sequence 
	$(x_1,x_2)\xrightarrow[]{s'}(x_3,x_4)$ in $\CCaLeft(\Scal)$ such that $\ell(s')=\s$.
	We only need to consider the most complex Case 4, all the other cases can be
	dealt with similarly. For Case 4, by definition, the corresponding 
	observable transition sequence is $(x_1,x_2)\xrightarrow[]{\varepsilon}(x_1,x_1)\xrightarrow[]{(t_1,t_2)}
	(x_5,x_6)\xrightarrow[]{s_1'}(x_7,x_8)\xrightarrow[]{\varepsilon}(x_7,x_7)\xrightarrow[]{s_2'}
	(x_3,x_4)$, where $s_1'(L)=s_1$, $s_1'(R)=s_2$, $s_2'(L)=s_3$, $s_2'(R)=s_4$. 

	\eqref{item12_spdet_FSA}, \eqref{item13_spdet_FSA}, and \eqref{item14_spdet_FSA} can be proved similarly.

	\eqref{item15_spdet_FSA} directly follows from definition.
\end{proof}

\begin{example}\label{exam2_spdet_FSA}
	Consider FSA $\Scal_3$, its detector $\Scal_{3det}$, and its ($\varepsilon$-extended) self-composition
	$\CCa(\Scal_3)$ ($\CCaLeft(\Scal_3)$) shown in Fig.~\ref{fig3_spdet_FSA}. There is a transition
	$\{x_1,x_2\}\xrightarrow[]{b}\{x_1,x_2\}$ in $\Scal_{3det}$, but there is neither transition sequence
	$(x_1,x_2)\xrightarrow[]{s'}(x_1,x_2)$ nor $(x_1,x_2)\xrightarrow[]{s'}(x_2,x_1)$ such that
	$\ell(s')=b$ in $\CCa(\Scal_3)$. However, in $\CCaLeft(\Scal_3)$, there is a transition sequence
	$(x_1,x_2)\xrightarrow[]{\varepsilon}(x_1,x_1)\xrightarrow[]{(t_3,t_4)}(x_1,x_2)$ such that
	$\ell(\varepsilon(t_3,t_4))=b$.
	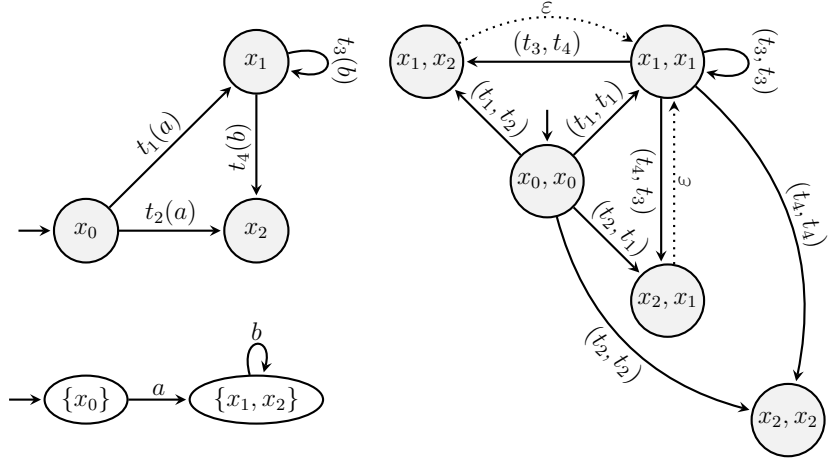
\begin{figure}[!htbp]
		\tikzset{global scale/.style={scale=#1,
    every node/.append style={scale=#1}}}
		\begin{center}
			\begin{tikzpicture}[global scale = 1.0,
				>=stealth',shorten >=1pt,thick,auto,node distance=2.8 cm, scale = 0.8, transform shape,
	->,>=stealth,inner sep=2pt,
				every transition/.style={draw=red,fill=red,minimum width=1mm,minimum height=3.5mm},
				every place/.style={draw=blue,fill=blue!20,minimum size=7mm}]
				\tikzstyle{emptynode}=[inner sep=0,outer sep=0]
				\node[state, initial, initial where = left] (x0) {$x_0$};
				\node[state] (x2) [right of = x0] {$x_2$};
				\node[state] (x1) [above of = x2] {$x_1$};

				\path[->]
				(x0) edge node [above, sloped] {$t_1(a)$} (x1)
				(x0) edge node [above, sloped] {$t_2(a)$} (x2)
				(x1) edge [loop right] node [above, sloped] {$t_3(b)$} (x1)
				(x1) edge node [above, sloped] {$t_4(b)$} (x2)
				;

				\node[elliptic state, initial, initial where = left] (x0') [below of = x0] {$\{x_0\}$};
				\node[elliptic state] (x12') [right of = x0'] {$\{x_1,x_2\}$};

				\path[->]
				(x0') edge node [above, sloped] {$a$} (x12')
				(x12') edge [loop above] node [above, sloped] {$b$} (x12')
				;

				\node[state] (x12'') [right of = x1] {$x_1,x_2$};
				\node[state, initial, initial where = above] (x00'') [below right of = x12''] {$x_0,x_0$};
				\node[state] (x11'') [above right of = x00''] {$x_1,x_1$};
				\node[state] (x21'') [below right of = x00''] {$x_2,x_1$};
				\node[state] (x22'') [below right of = x21''] {$x_2,x_2$};

				\path[->]
				(x00'') edge node [above, sloped] {$(t_1,t_2)$} (x12'')
				(x00'') edge node [above, sloped] {$(t_1,t_1)$} (x11'')
				(x00'') edge node [above, sloped] {$(t_2,t_1)$} (x21'')
				(x00'') edge [bend right] node [below, sloped] {$(t_2,t_2)$} (x22'')
				(x11'') edge [loop right] node [above, sloped] {$(t_3,t_3)$} (x11'')
				(x11'') edge node [above, sloped] {$(t_3,t_4)$} (x12'')
				(x11''.-100) edge node [below, sloped] {$(t_4,t_3)$} (x21''.100)
				(x12'') edge [dotted, bend left] node [above, sloped] {$\varepsilon$} (x11'')
				(x21''.80) edge [dotted] node [below, sloped] {$\varepsilon$} (x11''.-80)
				(x11'') edge [bend left] node [above, sloped] {$(t_4,t_4)$} (x22'')
				;
			\end{tikzpicture}
	\end{center}
	\caption{FSA $\Scal_3$ (upper left), its detector $\Scal_{3det}$ (lower left, the same as observer 
	$\Scal_{3obs}$), its 
	self-composition $\CCa(\Scal_3)$ (right, dotted transitions excluded), and its $\varepsilon$-extended
	self-composition $\CCaLeft(\Scal_3)$ (right).}
	\label{fig3_spdet_FSA}
	\end{figure}
\end{example}

With these properties, we are ready to give a new polynomial-time algorithm for verifying strong
periodic detectability by using $\CCaLeft(\Scal)$.

\begin{theorem}\label{thm8_spdet_FSA}
	An FSA $\Scal$ is
	not strongly periodically detectable if and only if in its $\varepsilon$-extended 
	self-composition $\CCaLeft(\Scal)$ as in \eqref{ConCom_Rever_Det_PN},
	at least one of the two following conditions holds.
	\begin{myenumerate}
		\item\label{item9_spdet_FSA}
			There is a reachable state $(x,\bar x)$ such that $x\ne\bar x$
			and there is a transition
			sequence $x\xrightarrow[]{s_1}x'\xrightarrow[]{s_2}x'$ in $\Scal$ for some
			$s_1\in (T_{uo})^*$, $s_2\in(T_{uo})^+$, $x'\in X$.
		\item\label{item10_spdet_FSA}
			There is a reachable transition cycle $(x_1,\bar x_1)\xrightarrow[]{s_1'}\cdots
			\xrightarrow[]{s_n'}(x_{n+1},\bar x_{n+1})$ for some positive integer $n$ such that
			$(x_1,\bar x_1)=(x_{n+1},\bar x_{n+1})$, $x_i\ne\bar x_i$, and $\ell(s_i')\in\Sig$
			for all $i\in\llb 1,n\rrb$.
	\end{myenumerate}
\end{theorem}

\begin{proof}
	We use Theorem~\ref{thm4_spdet_FSA} and Propositions~\ref{prop5_spdet_FSA} and
	\ref{prop8_spdet_FSA} to prove this result.

	We first check \eqref{item9_spdet_FSA} is equivalent to \eqref{item5_spdet_FSA}.

	``$\Rightarrow$'': Assume \eqref{item9_spdet_FSA} holds. 
	By \eqref{item15_spdet_FSA} of Proposition~\ref{prop8_spdet_FSA} and Proposition~\ref{prop5_spdet_FSA},
	for every reachable state $(x,x')$ of $\CCaLeft(\Scal)$ such that $x\ne x'$, either $\{x,x'\}\subset\Mt(\Scal,\ep)$
	or $\{x,x'\}$ is reachable in $\Scal_{det}$. Hence \eqref{item5_spdet_FSA} holds.

	``$\Leftarrow$'': Assume \eqref{item5_spdet_FSA} holds. If $q'=\Mt(\Scal,\ep)$, then \eqref{item9_spdet_FSA}
	holds. Otherwise (i.e., in case $|q'|=2$ and $q'\ne\Mt(\Scal,\ep)$), by \eqref{item11_spdet_FSA},
	\eqref{item12_spdet_FSA}, \eqref{item13_spdet_FSA}, \eqref{item14_spdet_FSA} of 
	Proposition~\ref{prop8_spdet_FSA}, one has \eqref{item9_spdet_FSA} holds.

	We second check \eqref{item10_spdet_FSA} is equivalent to \eqref{item6_spdet_FSA}.

	``$\Rightarrow$'': Assume \eqref{item10_spdet_FSA} holds. By \eqref{item15_spdet_FSA} of 
	Proposition~\ref{prop8_spdet_FSA} and the Pigeonhole Principle, there is a reachable transition cycle in 
	$\Scal_{obs}$
	none of whose states is a singleton. The by Proposition~\ref{prop5_spdet_FSA}, \eqref{item6_spdet_FSA} holds.

	``$\Leftarrow$'': Assume \eqref{item6_spdet_FSA} holds. By \eqref{item11_spdet_FSA} of 
	Proposition~\ref{prop8_spdet_FSA}, \eqref{item10_spdet_FSA} holds.
\end{proof}

Similarly to the case that Theorem~\ref{thm4_spdet_FSA} implies Theorem~\ref{thm5_spdet_FSA},
Theorem~\ref{thm8_spdet_FSA} also implies an $\NL$ upper bound for strong periodic
detectability of FSAs without any assumption.

\begin{example}\label{exam3_spdet_FSA}
	We next use one example to compare Theorem~\ref{thm4_spdet_FSA} with Theorem~\ref{thm8_spdet_FSA}.
	Reconsider $\Scal_3$ in Example~\ref{exam2_spdet_FSA} (shown in Fig.~\ref{fig3_spdet_FSA}, upper left).
	The existence of reachable transition cycle
	$\{x_1,x_2\}\xrightarrow[]{b}\{x_1,x_2\}$ in $\Scal_{3det}$ implies that $\Scal_3$ is not strongly
	periodically detectable by Theorem~\ref{thm4_spdet_FSA} (satisfying \eqref{item6_spdet_FSA}).
	The existence of reachable transition cycle
	$(x_1,x_2)\xrightarrow[]{\varepsilon}(x_1,x_1)\xrightarrow[]{(t_3,t_4)}(x_1,x_2)$ such that
	$\ell(\varepsilon(t_3,t_4))=b\in \Sig$ in $\CCaLeft(\Scal)$ also implies that $\Scal_3$ is not strongly
	periodically detectable, by Theorem~\ref{thm8_spdet_FSA} (satisfying \eqref{item10_spdet_FSA}).
\end{example}

\subsection{Verifying strong periodic D-detectability}
We also first characterize negation of $T_{\spec}$-strong periodic 
D-detectability. The following result directly follows from Definition~\ref{def11_Det_PN}.

\begin{proposition}\label{prop7_spdet_FSA}
	An FSA $\Scal$ is not
	strongly periodically D-detectable with respect to $T_{\spec}$
	if and only if for each positive integer $k$, there exist $s_k\in L^{\omega}({\Scal})$
	and $s'\sqsubset s_k$ such that for every $s''\in T^*$,
	$|\ell(s'')|<k$ and $s's''\sqsubset s$ imply $(\Mt(\Scal,\ell(s's''))\times\Mt(\Scal,\ell(s's'')))
	\cap T_{\spec}\ne\emptyset$.
\end{proposition}

Similarly to Theorem~\ref{thm3_spdet_FSA}, we can prove the following result.
We omit the similar proof.

\begin{theorem}\label{thm6_spdet_FSA}
	An FSA $\Scal$ is
	not strongly periodically D-detectable with respect to $T_{\spec}$
	if and only if in its observer $\Scal_{obs}$
	as in \eqref{observer_DES}, at least one of the two following conditions holds.
	\begin{myenumerate}
		\item\label{item7_spdet_FSA}
			There is a reachable state $q\in 2^X$ in $\Scal_{obs}$ and $x\in q$
			such that $(q\times q)\cap T_{\spec}\ne\emptyset$ and there is a transition
			sequence $x\xrightarrow[]{s_1}x'\xrightarrow[]{s_2}x'$ in $\Scal$ for some
			$s_1\in (T_{uo})^*$, $s_2\in(T_{uo})^+$, $x'\in X$.
		\item\label{item8_spdet_FSA}
			There is a reachable transition cycle such that
			each state $q$ of the cycle satisfies $(q\times q)\cap T_{\spec}\ne\emptyset$.
	\end{myenumerate}
\end{theorem}

\begin{theorem}\label{thm7_spdet_FSA}
	The problem of verifying strong periodic D-detectability with respect to $T_{\spec}$
	belongs to $\PSPACE$.
\end{theorem}

\begin{proof}
	Condition \eqref{item7_spdet_FSA} can be checked by guessing $q\in 2^X$, $x\in q$,
	and $x'\in X$ and doing the corresponding checks by nondeterministic research. 
	Since each state $q$ of $\Scal_{obs}$ is bounded by the number of states of $\Scal$,
	and $(q\times q)\cap T_{\spec}\ne\emptyset$ can be checked in time quadratic in the number 
	of states of $\Scal$, \eqref{item7_spdet_FSA} can be checked in $\NPSPACE$.

	Condition \eqref{item8_spdet_FSA} 
	can be checked by nondeterministically guessing a sequence 
	of label sequence and checking whether the sequence leads $\Scal_{obs}$ to such a transition
	cycle. Hence, \eqref{item7_spdet_FSA} can also be checked in $\NPSPACE$.

	Hence by Theorem~\ref{thm6_spdet_FSA}, the problem of verifying 
	strong periodic D-detectability with respect to 
	$T_{\spec}$ belongs to $\coNPSPACE$, i.e., $\PSPACE$. 
\end{proof}

\begin{remark}\label{rem2_spdet_FSA}
	One directly sees that the equivalent condition for strong periodic D-detectability
	of FSAs under Assumption~\ref{assum1_Det_PN} given in \cite[Theorem 9]{Shu2011GDetectabilityDES}
	(collected in Proposition~\ref{prop1_spdet_FSA})
	is exactly negation of \eqref{item8_spdet_FSA} in Theorem~\ref{thm6_spdet_FSA}. So the 
	algorithm induced from \cite[Theorem 9]{Shu2011GDetectabilityDES} usually does not
	work correctly without Assumption~\ref{assum1_Det_PN}. See the following example.
	
	Reconsider $\Scal_2$ (shown in
	Fig.~\ref{fig1_spdet_FSA}, right) and its observer $\Scal_{2obs}$ (shown in Fig.~\ref{fig2_spdet_FSA},
	right). As shown in Example~\ref{exam1_spdet_FSA}, $\Scal_2$ violates Assumption~\ref{assum1_Det_PN}.
	Now choose $T_{\spec}=\{(x_1,x_2)\}$. For every positive integer $k$, choose $s_k=t_2(t_4)^{\omega}
	\in L^{\omega}(\Scal_2)$, then for all $(t_4)^{n}$, where $n\ge0$, one has $t_2(t_4)^n\sqsubset s_k$,
	$|\ell( (t_4)^{n})|=0<k$, $\ell(t_2(t_4)^{n})=a$, and $(\Mt(\Scal_2,a)\times\Mt(\Scal_2,a))\cap T_{\spec}
	=T_{\spec}\ne\emptyset$. That is, $\Scal_2$ is not strongly periodically D-detectable
	with respect to $T_{\spec}$ by definition.
	However, since there is no cycle in $\Scal_{2obs}$, the condition 
	``every reachable transition cycle contains at least one state $q$ such that $(q\times q)\cap 
	T_{\spec}=\emptyset$'' in Proposition~\ref{prop1_spdet_FSA} is satisfied vacuously. Thus, $\Scal_{2obs}$
	is strongly periodically D-detectable with respect to $T_{\spec}$ by Proposition~\ref{prop1_spdet_FSA},
	which is incorrect.
\end{remark}

\begin{example}\label{exam4_spdet_FSA}
	We next illustrate Theorem~\ref{thm6_spdet_FSA}.
	Reconsider FSA $\Scal_3$ in Fig.~\ref{fig3_spdet_FSA} (upper left) and its observer $\Scal_{3obs}$
	in Fig.~\ref{fig3_spdet_FSA} (lower left). If we choose $T_{\spec}^1=\{(x_1,x_2)\}$, the existence of
	self-loop $\{x_1,x_2\}\xrightarrow[]{b}\{x_1,x_2\}$ in $\Scal_{3obs}$ satisfies \eqref{item8_spdet_FSA}
	of Theorem~\ref{thm6_spdet_FSA} (i.e., $(\{x_1,x_2\}\times\{x_1,x_2\})\cap T_{\spec}^1=T_{\spec}^1\ne\emptyset$),
	hence $\Scal_3$ is not strongly periodically D-detectable with respect
	to $T_{\spec}^1$. If we choose $T_{\spec}^2=\{(x_0,x_2)\}$, then by $\Scal_{3obs}$, one sees
	neither \eqref{item7_spdet_FSA} nor \eqref{item8_spdet_FSA} is satisfied, hence $\Scal_3$
	is strongly periodically D-detectable with respect to $T_{\spec}^2$.
\end{example}

\section{Conclusion}\label{sec4:conc}

In this paper, we 
obtained an $\NL$ upper bound for verifying strong periodic
detectability of FSAs without any assumption, strengthening the related results given in 
\cite{Shu2011GDetectabilityDES,Masopust2018ComplexityDetectabilityDES} under two assumptions
of deadlock-freeness and divergence-freeness.
We also obtained a $\PSPACE$ upper bound for verifying strong periodic D-detectability of FSAs
without any assumption, strengthening the related result given in 
\cite{Shu2011GDetectabilityDES,Balun2021ComplexitySPDdetJournal} also under the two assumptions.

As shown in our previous paper \cite{Zhang2019KDelayStrDetDES}, the self-composition 
method can be used to verify (delayed) strong detectability of FSAs without any assumption, 
but the detector method cannot. In this paper, we showed that both the detector method 
and a variant of the self-composition method can be
used to verify strong periodic detectability of FSAs without any assumption. It is an interesting
future topic to study the intrinsic relationships between the two methods.


\end{document}